\documentclass[11pt]{article}

\usepackage[utf8]{inputenc}
\usepackage[a4paper,margin=0.9in]{geometry}

\usepackage{amsmath,amssymb,amsfonts,amsthm,mathtools}
\usepackage{stmaryrd}

\usepackage{graphicx}
\usepackage{float}
\usepackage{array,booktabs,longtable}
\usepackage{caption}
\usepackage{subcaption}
\graphicspath{{./image/}}
\DeclareMathAlphabet{\mathcal}{OMS}{cmsy}{m}{n}

\usepackage{tikz}
\usepackage{tikz-cd}
\usetikzlibrary{arrows,positioning,matrix}
\usepackage[all,cmtip]{xy}
\usepackage{adjustbox}

\usepackage[inline]{enumitem}

\usepackage[noadjust]{cite}
\PassOptionsToPackage{hyphens}{url}
\usepackage{hyperref}

\usepackage{mathptmx}
\usepackage[dvipsnames]{xcolor}

\usepackage{todonotes}

\makeatletter
\newcommand{\inlineitem}[1][]{%
\ifnum\enit@type=\tw@
    {\descriptionlabel{#1}}
  \hspace{\labelsep}%
\else
  \ifnum\enit@type=\z@
       \refstepcounter{\@listctr}\fi
    \quad\@itemlabel\hspace{\labelsep}%
\fi}
\makeatother

\tikzcdset{scale cd/.style={every label/.append style={scale=#1},
    cells={nodes={scale=#1}}}}

\theoremstyle{definition}

\newtheorem{theorem}{Theorem}[section]
\newtheorem*{theorem*}{Theorem}

\newtheorem*{corollary*}{Corollary}
\newtheorem{lemma}[theorem]{Lemma}
\newtheorem{proposition}[theorem]{Proposition}

\newtheorem{definition}[theorem]{Definition}

\newtheorem{theoremL}{Theorem}

\theoremstyle{remark}
\newtheorem*{remark}{Remark}

\title{Brunnian spanning 3-disks for the 2-unlink in the 4-sphere}
\author{Weizhe Niu}
\date{}

\begin{document}
\maketitle

\begin{abstract}
We show that the $2$-component unlink in $S^4$ admits infinitely many isotopy classes of spanning $3$-disks that are Brunnian.
\end{abstract}

\section{Introduction}\label{sec:intro}

Let $L\subset S^4$ be an $m$-component $2$-link. A \emph{collection of spanning $3$-disks} for $L$ is a smoothly embedded union
$$
D=\bigsqcup_{i=1}^m D_i^3\subset S^4
\qquad\text{with}\qquad
\partial D=L.
$$
Let $U_m$ denote the standard unlink with $m$ components in $S^4$. The complement of an open tubular neighbourhood of $U_m$ is diffeomorphic to
$\#_m S^1\times D^3$ which we write as $\#_{i=1}^m A_i$ with $A_i=S^1\times D^3$. In the complement, there is a standard family of spanning disks
$$
D_m^{\mathrm{std}}=\{D_i\}_{i=1,\dots,m},
\qquad
D_i=\{pt\}\times D^3\subset A_i.
$$
For $m=1$, \cite{Budney-gabai} provides infinitely many non-isotopic spanning disks for the unknot. We focus on the case $m=2$ and write $D^{\mathrm{std}}:=D_2^{\mathrm{std}}$.

\begin{figure}
    \centering
    \includegraphics[width=0.5\linewidth]{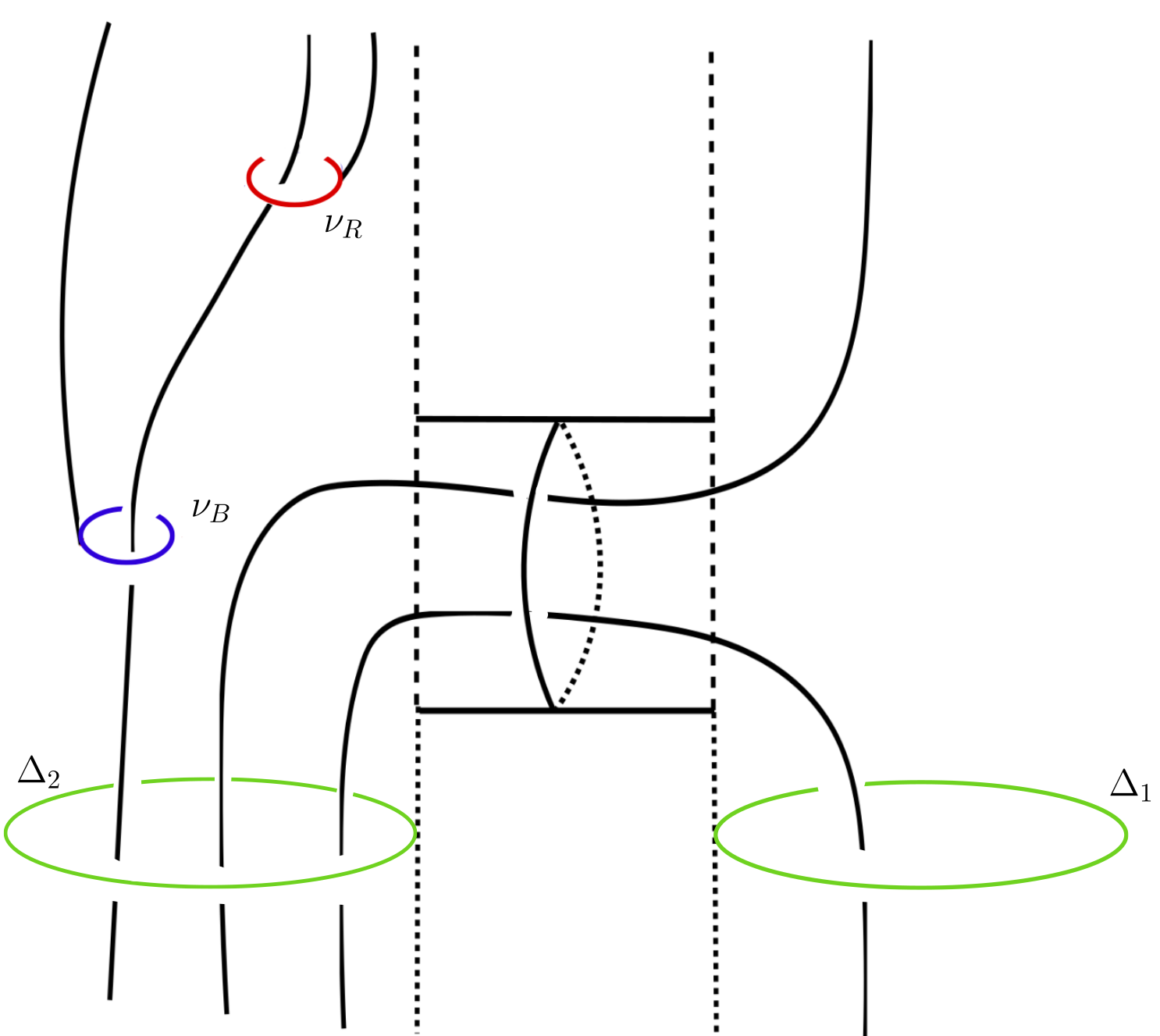}
    \caption{The barbell $\mathcal{B}(t\nu_B\nu_R tut^{-1})$ embedded in $\natural_2 S^1\times D^3\subset \#_2 S^1\times D^3$.}
    \label{extra}
\end{figure}

In this paper, we construct infinitely many pairs of non-isotopic spanning disks for the unlink $U_2$ that are \textbf{Brunnian} in the sense that each of the two disks is itself isotopically standard. They arise as images of the standard disks under the barbell diffeomorphisms
$$
\Phi_{\mathcal{B}(t\nu_B\nu_R tu^kt^{-1})}\in \pi_0\mathrm{Diff}(\#_2S^1\times D^3,\partial),
\qquad
k\in \mathbb{Z}^{+},
$$
defined in \cite{niu2025mappingclassgroup4dimensional}
(cf.\ Proposition 5.12 of \cite{niu2025mappingclassgroup4dimensional}; see Section~\ref{sec:dax} for a brief review of the barbell picture). See Figure \ref{extra} for the case $k=1$. The key input is the generalized $W_3$ invariant from \cite{niu2025mappingclassgroup4dimensional}:
$$W_3^{\Delta}\colon \pi_0\mathrm{Diff}(\natural_2 S^1\times D^3,\partial)\to \Lambda\coloneqq\mathbb{Q}\langle t_1,t_3,u_1,u_3\rangle/\mathcal{H}$$
which depends on a choice of a properly
embedded $3$-ball $\Delta\subset X=\natural_2 S^1\times D^3$. Here $\mathcal{H}$ denotes the Hexagon relations defined in Theorem 3.16 of \cite{niu2025mappingclassgroup4dimensional}; see \ref{eq:HexagonElement}). We take
$$
\Delta_1=D_1,\qquad \Delta_2=D_2,
$$
and write the corresponding invariants as $W_3^{\Delta_1}$ and $W_3^{\Delta_2}$. Previously, the generalization of $W_3$ is mainly seen in boundary-connected sums, see for example \cite{lin2025mappingclassgroups4manifolds} and \cite{niu2025mappingclassgroup4dimensional}. Using the relationship between
$$
X=\natural_2 S^1\times D^3
\qquad\text{and}\qquad
Y=\#_2 S^1\times D^3,
$$
we are able to obtain an induced invariant $W_3'^{\Delta_i}$ on a subgroup of $\pi_0\mathrm{Diff}(Y,\partial)$ (Section~\ref{sec:setup}) which we use to detect the barbell diffeomorphisms in an internal connected sum $Y$. The main results of this paper are:

\begin{theoremL}\label{thm:A}
For $k\ge 1$, the barbell diffeomorphisms
$\Phi_{\mathcal{B}(t\nu_B\nu_R tu^kt^{-1})}\in \pi_0\mathrm{Diff}(\#_2S^1\times D^3,\partial)$ are non-trivial and pairwise non-isotopic. Moreover, they are detected by the induced invariants $W_3'^{\Delta_i}$ (equivalently, by
$W_3^{\Delta_i}$ via the construction of Section~\ref{sec:setup}).
\end{theoremL}

\begin{theoremL}\label{cor:B}
The unlink $U_2$ admits infinitely many collections of Brunnian spanning disks
$$
\Phi_{\mathcal{B}(t\nu_B\nu_R tu^kt^{-1})}(D^{\mathrm{std}})\qquad (k=1,2,\dots),
$$
which are pairwise non-isotopic.
\end{theoremL}

To see these collections are Brunnian, observe that capping off $A_i$ by gluing a copy of $S^2\times D^2$ gives rise to 
\begin{enumerate}
    \item For $i=1$, the resulting barbell is isotopic to $\mathcal{B}(t\nu_B\nu_R)$.
    \item For $i=2$, the resulting barbell is isotopic to $\mathcal{B}(\nu_B\nu_Ru^k)$.
\end{enumerate}
In both cases, the above barbells are trivial in $S^1\times D^3$ (cf.~Section 2 of \cite{niu2025mappingclassgroup4dimensional}).

\begin{remark}
    These spanning disks become isotopically standard when viewed in $S^4\subset D^5$ since the barbells $\mathcal{B}(t\nu_B\nu_R tu^kt^{-1})$ can be unlinked. This agrees with the result by \cite{powell2025spanning3discs4spherepushed}.
\end{remark}

\medskip
The paper is organized as follows. Section~\ref{sec:setup} relates spanning disks to diffeomorphisms of $Y$ and constructs the induced invariant $W_3'^{\Delta_i}$. Section~\ref{sec:dax} recalls the Dax invariant and explains how the relevant embedding subgroup produces barbell generators; in particular,
the corresponding $W_3$-values lie in the admissible span of Proposition~\ref{prop:admissible-span}.
Section~\ref{sec:computations} constructs coefficient functionals on the $W_3$-target (modulo hexagon relations) which vanish on the admissible span but detect the family $t\nu_B\nu_Rtu^kt^{-1}$, completing the proof of Theorem~\ref{thm:A}. Finally, Section~\ref{thbproof} proves Theorem \ref{cor:B}.
\vspace{3mm}
\par\noindent\textbf{Acknowledgement.} The author thanks Mark Powell for drawing his attention to this question, and Jianfeng Lin and Brendan Owens for helpful comments.

\section{From spanning disks to embedded $3$-balls in $X$}\label{sec:setup}

Let
$$
X=S^1\times D^3\natural S^1\times D^3,
\qquad
Y=S^1\times D^3\# S^1\times D^3.
$$
Figure~\ref{boundarysumrelation} schematically depicts the relation between $X$ and $Y$ (drawn one dimension lower).
Let $\mathcal{U}$ be a standard properly embedded arc in $Y$ as depicted in Figure~\ref{boundarysumrelation}. Then $X$ is obtained from $Y$
by drilling out a neighbourhood of $\mathcal{U}$. Conversely, attaching a $3$-handle to $X$ along a neighbourhood of the boundary 2-sphere of a slice of the connected sum neck gives $Y$. The green part represents this attachment (drawn as a $3$-dimensional $2$-handle) in Figure \ref{boundarysumrelation}.

\begin{figure}[t]
    \centering
    \includegraphics[width=0.42\textwidth]{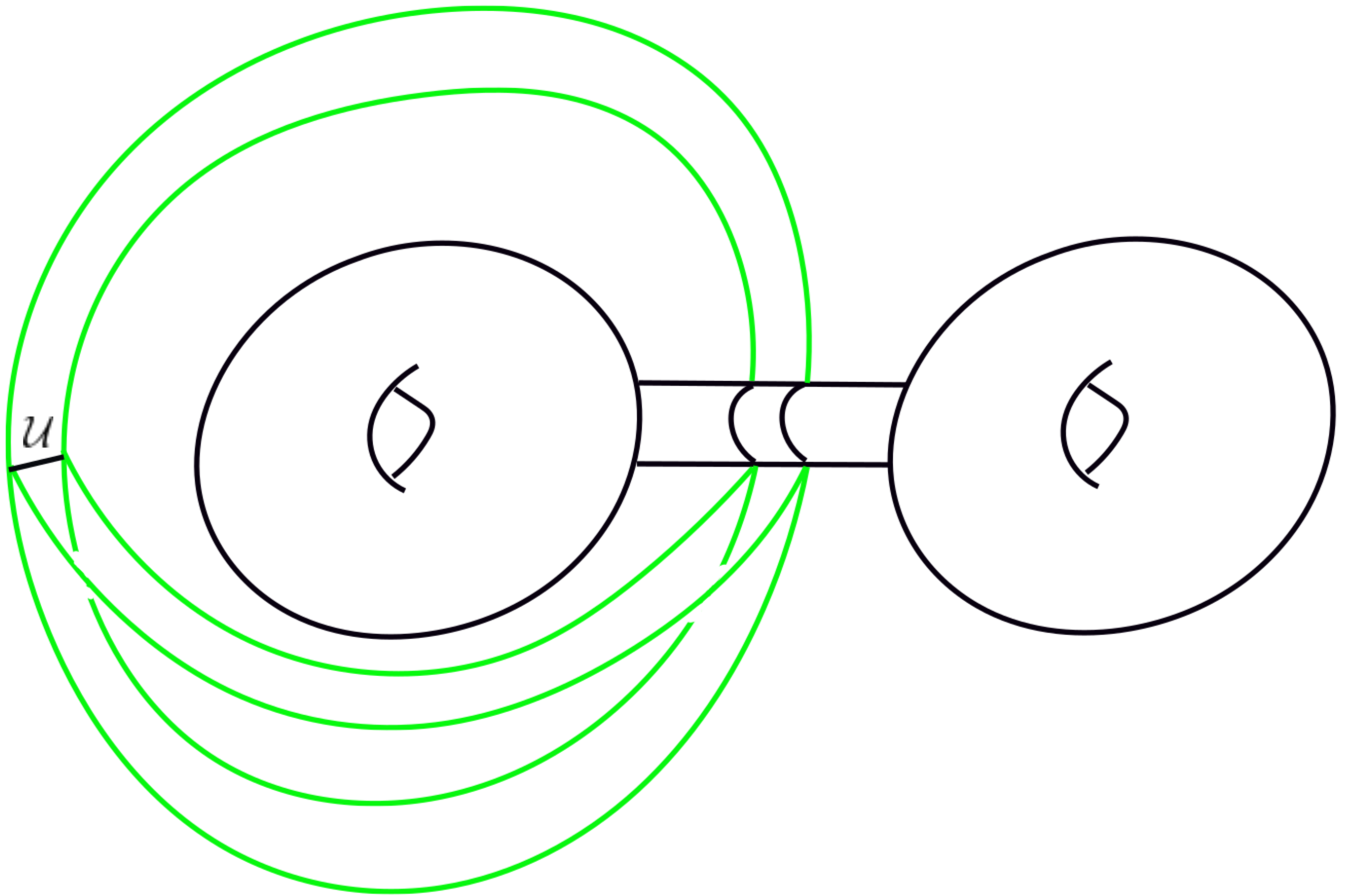}
    \caption{The relationship between $S^1\times D^3 \natural S^1\times D^3$ and $S^1\times D^3 \# S^1\times D^3$.}
    \label{boundarysumrelation}
\end{figure}

We denote by $\mathrm{Emb}(I\times D^3,Y;\mathcal{U}\times D^3)$ the path component of the space of properly embedded copies of $I\times D^3$ in $Y$ which agree with $\mathcal{U}\times D^3$ on $\{0,1\}\times D^3$, and which contains $\mathcal{U}\times D^3$.
Consider the fibration given by restricting a diffeomorphism of $Y$ to the fixed neighbourhood $\mathcal{U}\times D^3$:
$$
\mathrm{Diff}(X,\partial)\longrightarrow \mathrm{Diff}(Y,\partial)
\longrightarrow \mathrm{Emb}(I\times D^3,Y;\mathcal{U}\times D^3).
$$
The last few terms of the induced long exact sequence of homotopy groups are:
$$
\dots \to \pi_1\mathrm{Emb}(I\times D^3,Y;\mathcal{U}\times D^3)\xrightarrow{p}
\pi_0\mathrm{Diff}(X,\partial)\xrightarrow{e}\pi_0\mathrm{Diff}(Y,\partial)
\xrightarrow{r}\pi_0\mathrm{Emb}(I\times D^3,Y;\mathcal{U}\times D^3)\to \dots,
$$
where the map $p$ is given by isotopy extension and the map $e$ is given by extending diffeomorphisms of $X$ to $Y$ using the identity. For each $i=1,2$, we summarize the situation in the diagram below. Here
$$
\pi_0 \mathrm{Emb}(D^3,X)\coloneqq\pi_0 \mathrm{Emb}(D^3,X;\partial\Delta_i),
$$
denotes the set of isotopy classes of embedded $3$-balls in $X$ whose boundaries coincide with $\partial \Delta_i$, and similarly for $\pi_0 \mathrm{Emb}(D^3,Y)$. The map $l$ is the natural inclusion. We drop $\partial$ in the diagram to save space, but all mapping class groups shown in the diagram are boundary-preserving. The maps $s$ and $\phi$ are given by applying diffeomorphisms to $\Delta_i$, and the lower left map $q$ is the quotient map. The composition $d\circ s$ is $W_3^{\Delta_i}$.

\adjustbox{scale=1,center}{
\begin{tikzcd}
 \pi_1\mathrm{Emb}(I\times D^3,Y;\mathcal{U}\times D^3)\arrow{r}{p}\arrow{d}{W_3^{\Delta_i} \circ p} &
\pi_0\mathrm{Diff}(X)\arrow{r}{e}\arrow{d}{s} &
\pi_0\mathrm{Diff}(Y)\arrow{r}{r}\arrow{d}{\phi} &
\pi_0\mathrm{Emb}(I\times D^3,Y;\mathcal{U}\times D^3)
\\
  \Lambda=\mathbb{Q}\langle t_1,t_3,u_1,u_3\rangle/\mathcal{H} \arrow{d}{q} &
\pi_0\mathrm{Emb}(D^3,X)\arrow{l}{d}\arrow{r}{l} &
\pi_0\mathrm{Emb}(D^3,Y) &
\\
 \Lambda/(W_3^{\Delta_i} \circ p)\bigl(\pi_1\mathrm{Emb}(I\times D^3,Y;\mathcal{U}\times D^3)\bigr)
& & &
\ker(r)\arrow[hook]{uul}\arrow{ul}{\phi|_{\ker(r)}}\arrow{lll}{W_3'^{\Delta_i}}
\end{tikzcd}
}

There is an induced invariant 
$$W_3'^{\Delta_i}\colon \mathrm{ker}(r)\to \Lambda/(W_3^{\Delta_i}\ \circ p)\bigl(\pi_1\mathrm{Emb}(I\times D^3,Y;\mathcal{U}\times D^3)\bigr)$$
defined as follows. For a class in $\mathrm{ker}(r)\subset\pi_0\mathrm{Diff}(Y,\partial)$ lift it along $e$ to $\pi_0\mathrm{Diff}(X,\partial)$, and apply $q\circ W_3^{\Delta_i}$.
This is well-defined: any two choices of lifts to $\pi_0\mathrm{Diff}(X)$ differ by an element in $p(\pi_1 \mathrm{Emb}(I\times D^3,Y;\mathcal{U}\times D^3))$, hence they differ by something whose $W_3^{\Delta_i}$ value lies in the subgroup we are quotienting by. This shows independence of the choice of lift. In particular, for $[\Phi]\in \mathrm{ker}(r)$, if the composition $W_3'^{\Delta_i}([\Phi]) \neq 0$, then $[\Phi]$ is non-trivial in $\pi_0\mathrm{Diff}(Y,\partial)$.

In Sections \ref{sec:dax} and \ref{sec:computations}, we prove that
\[W_3'^{\Delta_i}(\Phi_{\mathcal{B}(t\nu_B\nu_R tu^kt^{-1})})\neq 0\]
for $k\in\mathbb{Z}^{+}$ and $i=1,2$ which leads to Theorem \ref{thm:A}. In Section \ref{thbproof}, we prove that $\Phi_{\mathcal{B}(t\nu_B\nu_R tu^kt^{-1})}(\Delta_i)$ produces isotopically non-trivial spanning disks.



\section{The Dax invariant and barbell generators}\label{sec:dax}

In this section, we analyse the image $p(\pi_1 \mathrm{Emb}(I\times D^3,Y))$. Since the map $\mathrm{Emb}(I\times D^3,Y)\to \mathrm{Emb}(I,Y)$ given by restricting to $I\times \{0\}$ is a fibration with fibre homotopy
equivalent to the free loop space of $SO(3)$, it suffices to understand $\pi_1\mathrm{Emb}(I,Y)$. We use Theorem 1.1 in \cite{danica} that provides an isomorphism 
\[\mathbb{Z}[\pi_1 Y\setminus 1]/\mathrm{dax}_\mathcal{U}(\pi_3 Y) \to \pi_1( \mathrm{Emb}(I,Y;\mathcal{U})).\]
Here $\mathbb{Z}[\pi_1 Y\setminus 1]$ is viewed as simply an abelian group and we ignore the ring
structure. In fact, for any smooth oriented 4-manifold $Y$, there is a central group extension
$$
1\to \mathbb{Z}[\pi_1 Y\setminus 1]/\mathrm{dax}_\mathcal{U}(\pi_3 Y)\to \pi_1( \mathrm{Emb}(I,Y;\mathcal{U}))\to \pi_2 Y\to 1
$$
where $\mathrm{dax}_\mathcal{U}$ denotes the \textbf{dax invariant}:
$$
\mathrm{dax}_\mathcal{U}\colon \pi_3 Y \to \mathbb{Z}[\pi_1 Y\setminus 1]
$$ which we briefly recall below. Here central means the image of $\mathbb{Z}[\pi_1 Y\setminus 1]/\mathrm{dax}_\mathcal{U}(\pi_3 Y)$ lies in the centre of
$\pi_1( \mathrm{Emb}(I,Y;\mathcal{U}))$.

For $a\in \pi_3 Y$, $\mathrm{dax}_\mathcal{U}(a)$ is defined by picking a 2-parameter family of immersions of $I$ into $Y$ representing $a$, which can be
viewed as a self-homotopy of the constant loop at $\mathcal{U}$, and analyzing the double points.
In other words, choose a map $F\colon S^2\to \mathrm{Imm}(I,Y;\mathcal{U})$ such that the composition with the concatenation map with
$\mathcal{U}^{-1}$, the inverse arc of $\mathcal{U}$, gives $a$.
The space $\mathrm{Imm}(I,Y;\mathcal{U})$ denotes the space of immersed intervals with the same endpoints as $\mathcal{U}$.
There is a natural inclusion map $\mathrm{Emb}(I,Y;\mathcal{U})\hookrightarrow \mathrm{Imm}(I,Y;\mathcal{U})$.
The map $\overline{F}\colon I^2\times I\to I^2\times Y$ defined by $(t,\theta)\to (t,F(t)(\theta))$ can be perturbed to an immersion with
only (finitely many) isolated transverse double points.
For each double point $(t_i,x_i)$ with $x_i=F(t_i)(\theta_{-})=F(t_i)(\theta_{+})$ where $\theta_{-}<\theta_{+}$, define a double point loop
$g_{x_i}$ as the concatenation $F(t_i)|_{[-1,\theta_{-}]}\cdot F(t_i)^{-1}|_{[-1,\theta_{+}]}$.
Then
$$
\mathrm{dax}_\mathcal{U}(a)\coloneqq\sum \epsilon_{x_i} g_{x_i} \in \mathbb{Z}[\pi_1 Y]
$$
where $\epsilon_{x_i}$ is the local orientation of $\overline{F}$ obtained by comparing the orientations of tangent bundles to the image of the
derivatives of $\overline{F}$ with the tangent space of $I^2\times Y$. Informally, it records the signed fundamental-group elements associated to double points occurring in a generic 2-parameter family of immersed arcs. See \cite{danica} for details.

Let $\pi_1(Y,\partial Y) $ be the set of homotopy classes of maps $k\colon I^1\to Y$ with endpoints $k(-1)=x_{-}$ and $k(1)\in \partial Y$.
It admits an action by $\pi_1 (Y,x_{-})$ via pre-concatenation. The following theorem from \cite{danica} relates
$\mathrm{dax}_{\mathcal{U}}$ and the \textbf{equivariant intersection pairing}
$$
\lambda\colon \pi_3 Y\times \pi_1 (Y,\partial Y)\to \mathbb{Z}[\pi_1 Y\setminus 1].
$$

\begin{theorem}(See Theorem A in \cite{danica})
\label{formi}
Let $\mathcal{U}_{-}\colon I\to Y$ be an embedding of an arc with endpoints $\mathcal{U}_{-}(-1)=x_{-}$ and
$\mathcal{U}_{-}(1)=x_{-}'$ where $x_{-}'$ is a point close to $x_{-}$ such that $\mathcal{U}_{-}$ is isotopic relative to the endpoints
into $\partial Y$. For a properly embedded arc $\mathcal{U}$ in $Y$ and for $a\in \pi_3 Y$ and $g\in \pi_1 Y$, the following is true:
    \begin{itemize}
        \item $\mathrm{dax}_\mathcal{U}(a)=\mathrm{dax}_{{\mathcal{U}}_{-}}(a)+\lambda(a,\mathcal{U}),$
        \item $\mathrm{dax}_{{\mathcal{U}}_{-}}(ga)=g\mathrm{dax}_{{\mathcal{U}}_{-}}(a)\Bar{g}-\lambda(ga,g)+\lambda(g,ga)$.
    \end{itemize}
\end{theorem}
Here $\mathrm{dax}_{{\mathcal{U}}_{-}}$ is defined in the same way as $\mathrm{dax}_\mathcal{U}$ with a different basepoint.

We now briefly recall the definition of the equivariant intersection pairing $\lambda$ and refer the reader to \cite{danica}.
For $a\in \pi_3 Y$ represented by $A\colon S^3\to Y$ and $k\in \pi_1(Y,\partial Y)$ represented by $k\colon (I,\partial I)\to (Y,\partial Y)$,
we can assume that they intersect transversely in the interior of $Y$.
For each intersection point $y$, there is a double point loop $\lambda_y(A,k)=\lambda_y(A)\cdot \lambda_y(k)^{-1}$ where
$\lambda_y(A)$ and ${\lambda_y(k)}^{-1}$ are paths from $k(-1)=x_{-}$ to $y$ along $A$ and $k$ respectively.
Define
$$
\lambda(a,k)=\sum _{y\in (A\cap k)\setminus \{x_{-}\}} \epsilon_{y}(A,k)[\lambda_y(A,k)]/[1],
$$
where the sign $\epsilon_y(A,k)$ is given by the local orientation at $y$.
One verifies that $\lambda$ is linear in the first coordinate and $\lambda(a,gk)=\lambda(a,g)+\lambda(a,k)\Bar{g}$ for $g\in \pi_1 Y$,
$k\in \pi_1 (Y,\partial Y)$ and $a\in \pi_3 Y$.
Note that the quotient by $[1]$ means that we forget the term at $1\in \pi_1 Y$.

\begin{figure}[t]
    \centering
    \includegraphics[width=0.52\textwidth]{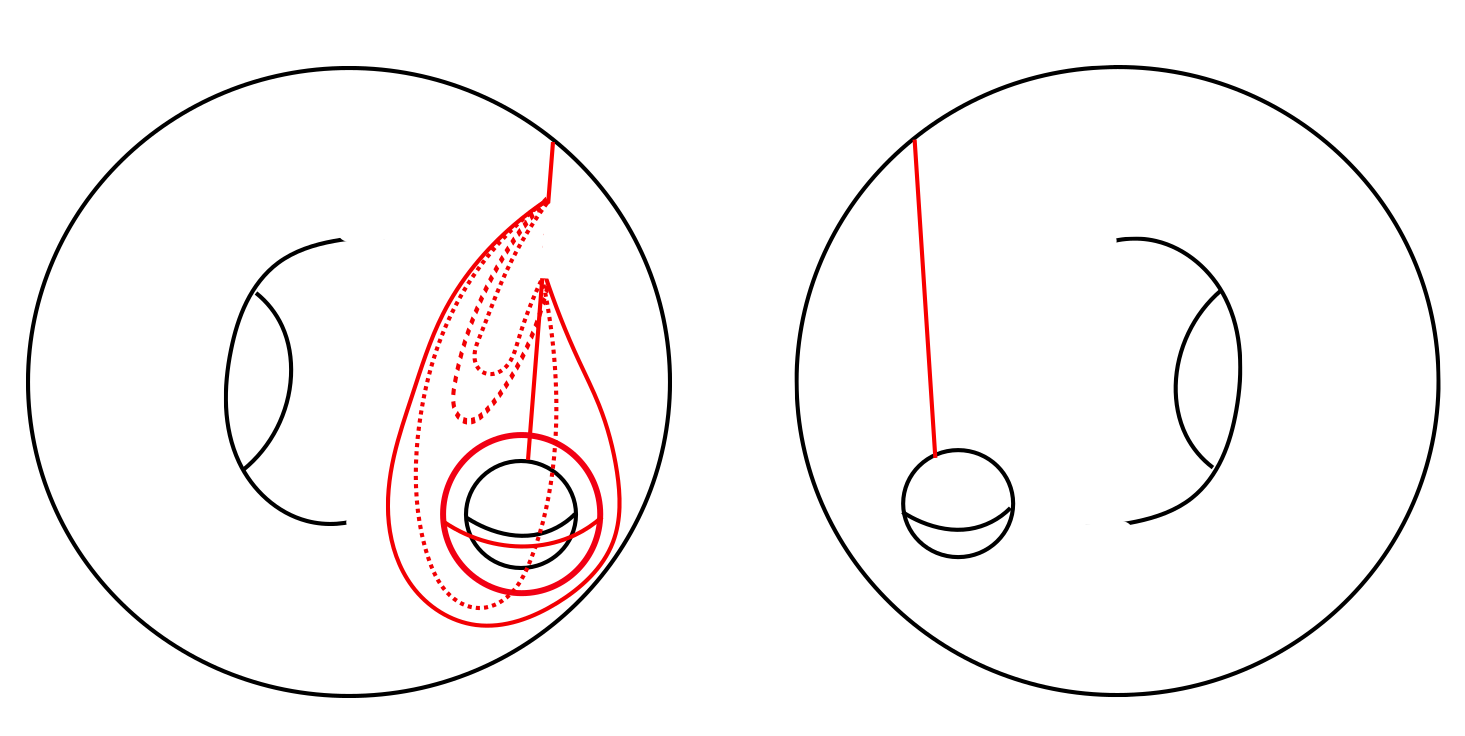}
    \caption{Representation of the connected-sum sphere $S$ by a loop of embedded intervals.}
    \label{dax1}
\end{figure}

However, for us, the situation is simplified for $Y=S^1\times D^3\#S^1\times D^3$.
The group $\pi_3 Y\cong \mathbb{Z}[t^{\pm},u^{\pm}]$ is generated by the action of
$\pi_1 Y\cong\mathbb{Z}*\mathbb{Z}=\langle u \rangle * \langle t \rangle$ on the connected sum sphere $S$.
Furthermore, $\pi_2 Y=0$, thus there is an isomorphism
$$
\mathbb{Z}[\pi_1 Y\setminus 1]/\mathrm{dax}_\mathcal{U}(\pi_3 Y)\to \pi_1( \mathrm{Emb}(I,Y;\mathcal{U})).
$$
Figure~\ref{dax1} shows the 3-dimensional analog representation of the connected-sum 2-sphere by a loop of arcs.

Using Theorem~\ref{formi}, we can calculate the image of $\lambda$ which is given by
$2 \mathbb{Z}[t^{\pm 1 },u^{\pm 1}\setminus 1]$ as follows.
For a representative of $1\in \pi_3 Y$, say the connected-sum sphere, there are no self-intersections
(see Figure~\ref{boundarysumrelation}), contributing to a trivial loop and hence the equivariant intersection pairing is 0. For a representative $a$ of a polynomial, for example $t^i u^jt^k\in\pi_3 Y$,  $\lambda(a,\mathcal{U})$ gives back $t^i u^jt^k$.
Further, if we choose $\mathcal{U}_{-}$ to be an arc with both endpoints in the same boundary component of
$Y$, then $\mathrm{dax}_{\mathcal{U}_{-}}(t^i u^jt^k)=t^i u^jt^k$, since we can choose a
representative of it with a single double point that produces the same polynomial.
It follows that $\pi_1 \mathrm{Emb}(I,Y)$ is isomorphic to
$$
\mathbb{Z}[t^{\pm 1 },u^{\pm 1}\setminus 1]/2 \mathbb{Z}[t^{\pm 1 },u^{\pm 1}\setminus 1]
\cong (\mathbb{Z}/2) [t^{\pm 1 },u^{\pm 1}\setminus 1].
$$
An element here is described by picking a small sub-arc of the basepoint (see Figure~\ref{boundarysumrelation}) for the basepoint arc
$\mathcal{U}$, pushing it along a path which corresponds to this element, and spinning around the basepoint arc, and finally coming back.
This process is defined as \textbf{spinning} in Definition 4.1 of \cite{Budney-gabai}. It is also described in \cite{danica}.
\begin{figure}
  \centering
  \begin{minipage}{0.43\textwidth}
    \includegraphics[width=\textwidth]{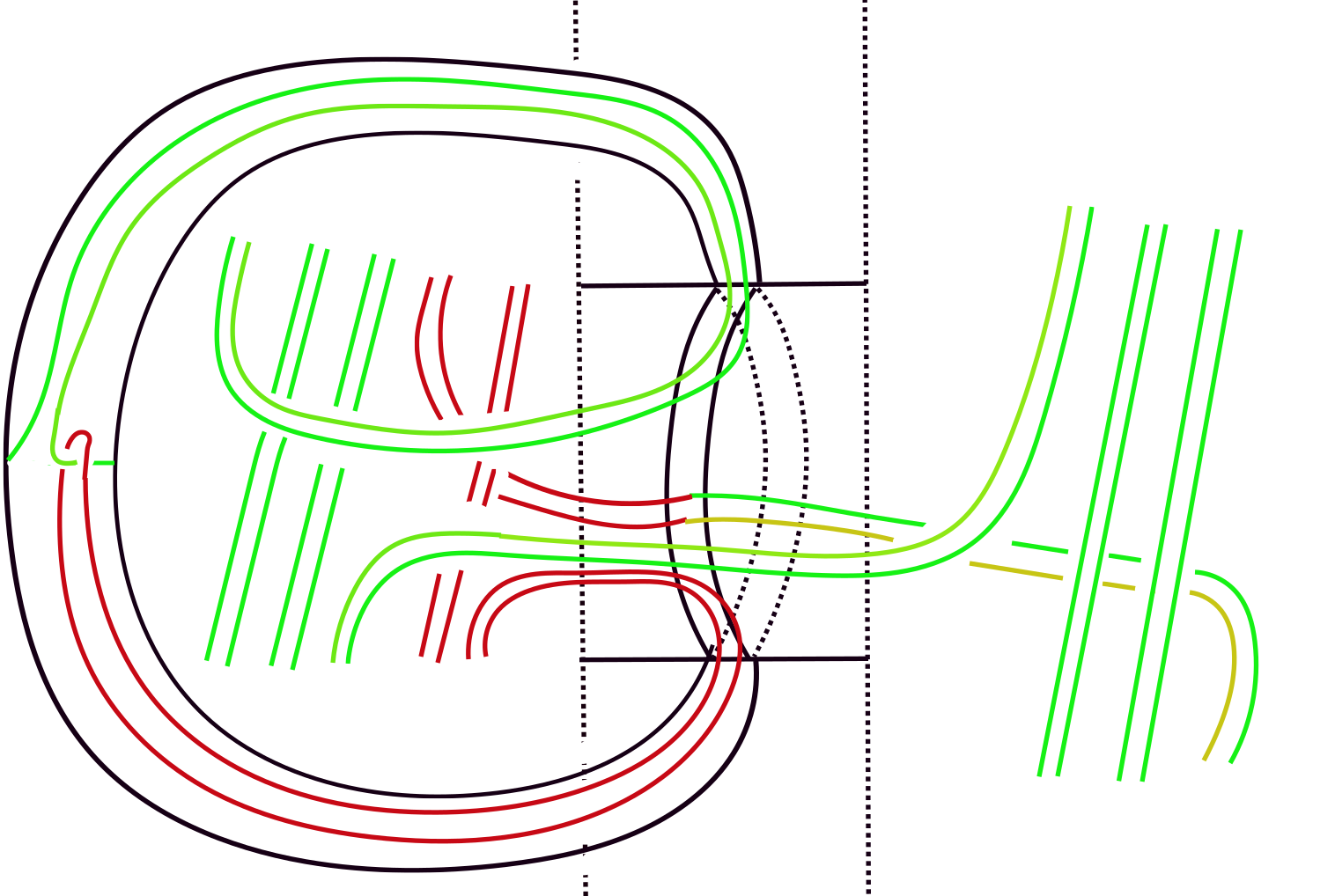}
    \caption{The element $t^3 u^3 t^2$ in $\pi_1\mathrm{Emb}(I,X)$.}
    \label{induce}
  \end{minipage}
  \hfill
  \begin{minipage}{0.43\textwidth}
    \includegraphics[width=\textwidth]{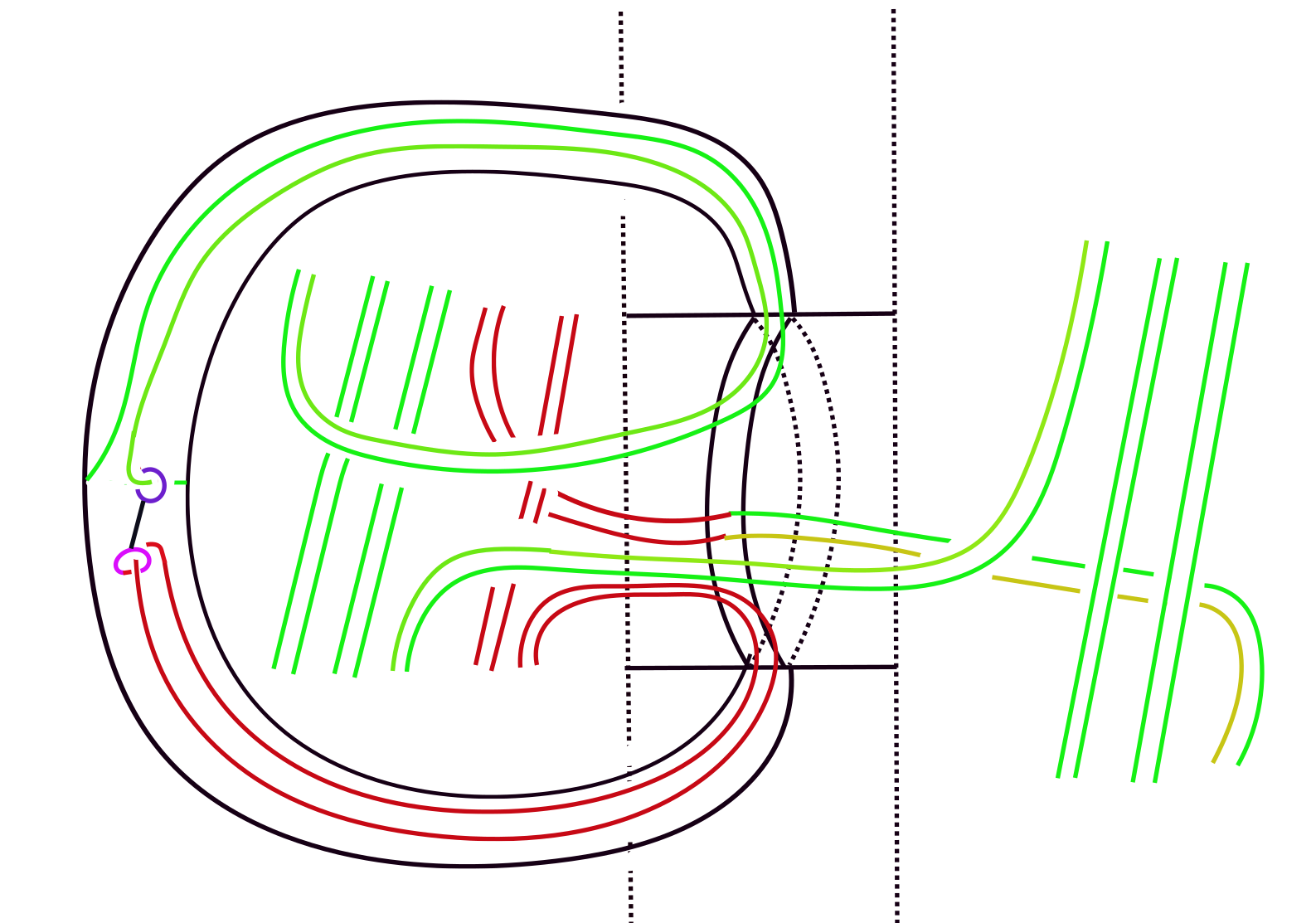}
    \caption{The embedded barbell induced by $t^3 u^3 t^2$ in $\pi_1\mathrm{Emb}(I,X)$.}
    \label{induce2}
  \end{minipage}
\end{figure}

\begin{figure}
  \centering
  \begin{minipage}{0.43\textwidth}
    \includegraphics[width=\textwidth]{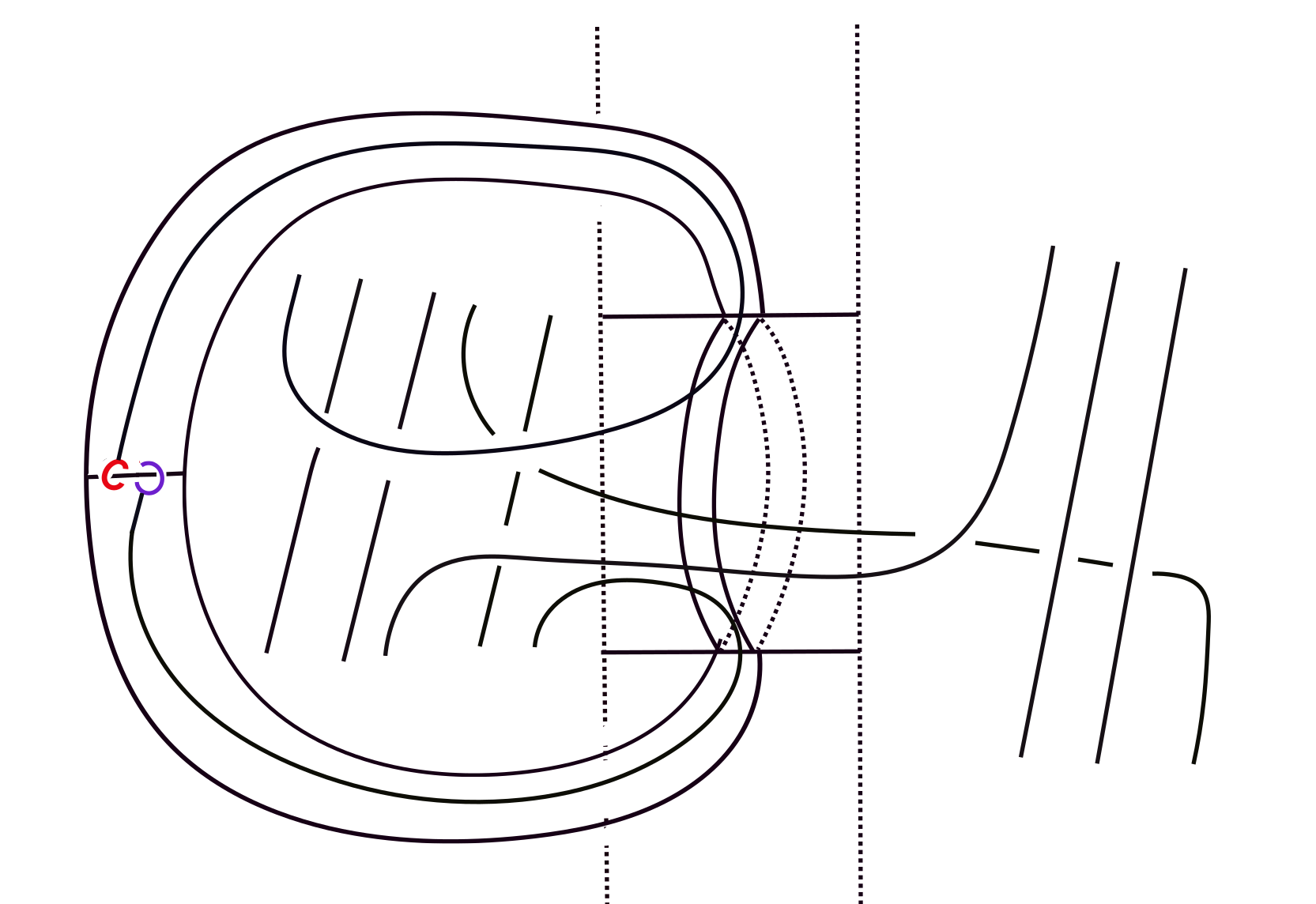}
    \caption{The induced barbell by $t^3 u^3 t^2$ dragged to a standard position.}
    \label{induce3}
  \end{minipage}
  \hfill
  \begin{minipage}{0.43\textwidth}
    \includegraphics[width=\textwidth]{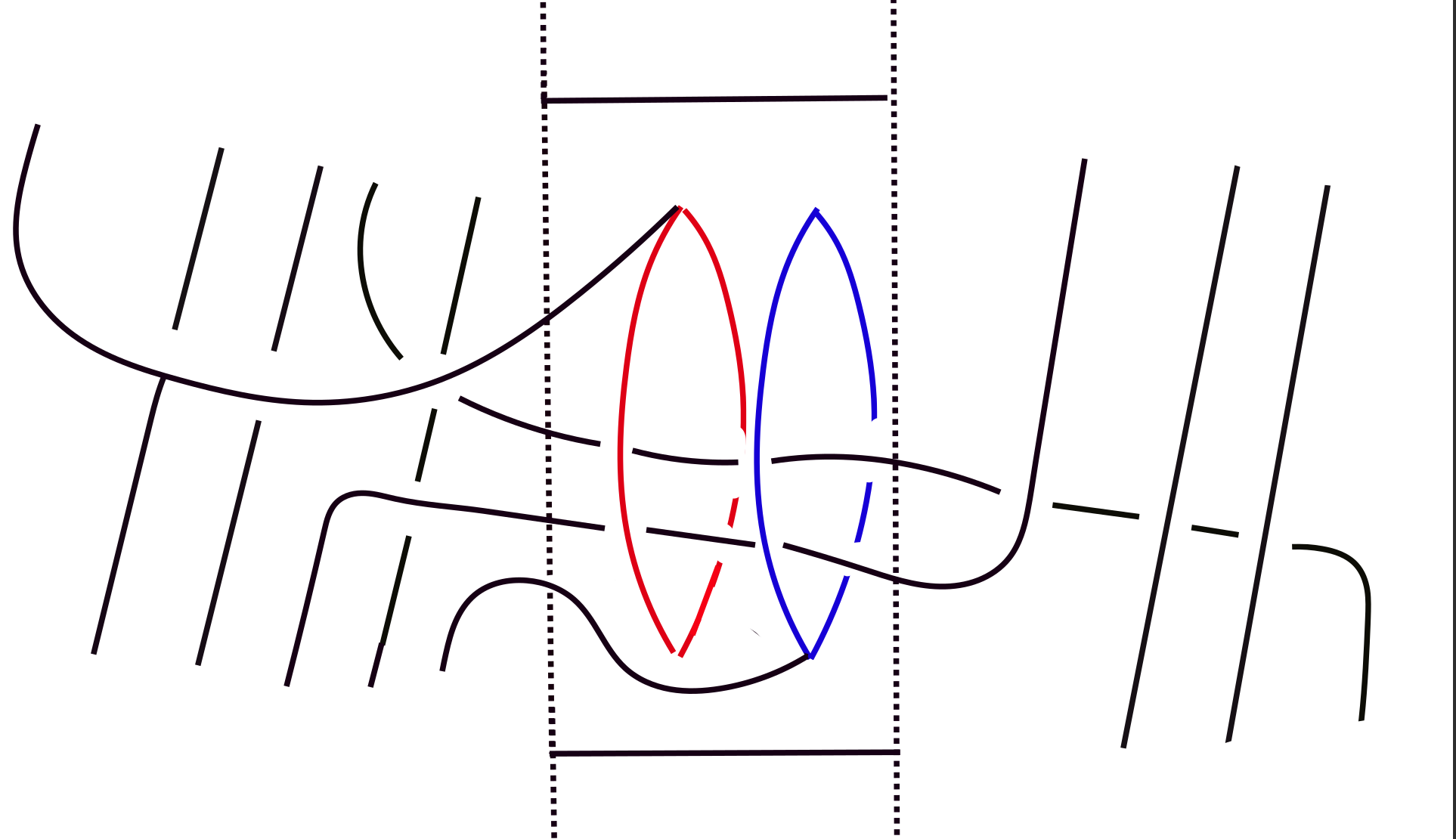}
    \caption{The induced barbell $\mathcal{B}(t^{-2}\nu_R^{-1}\nu_B^{-1}u^{-3}\nu_B\nu_R t^{-3} \nu_B)$ in standard position in $X\subset Y$.}
    \label{induce4}
  \end{minipage}
\end{figure}
The image $p(\pi_1 \mathrm{Emb}(I,Y))$ can be described by \textbf{barbell diffeomorphisms}
(cf.\ \cite{Budney-gabai}, \cite{Budney-gabai2}, \cite{niu2025mappingclassgroup4dimensional}) specified by words in the free group of two
generators.
Figure~\ref{induce} is an example of an element in $\pi_1\mathrm{Emb}(I,X)$ that corresponds to $t^3u^3t^2$, where a subarc of $\mathcal{U}$
spins around the loop $t^3u^3t^2$ and links itself. This creates a double-point whose resolution gives rise to a 1-parameter family of embedded
arcs. Using Section 4 of \cite{Budney-gabai}, it follows that the isotopy extension leads to a diffeomorphism supported in a neighbourhood diffeomorphic to a model barbell manifold.
Figure~\ref{induce2} depicts this initial barbell manifold induced from the double point resolution.
Figure~\ref{induce3} is the result of an isotopy that puts both cuff spheres in a standard position and Figure~\ref{induce4} is obtained by
drilling out a neighbourhood of $\mathcal{U}$ to get back to $X$ and corresponds to
$$
t^{-2}\nu_R^{-1}\nu_B^{-1}u^{-3}\nu_B\nu_R t^{-3} \nu_B.
$$

More generally, we observe that the image $p(\pi_1 \mathrm{Emb}(I,Y))$ is generated by diffeomorphisms induced from such barbells:
$$
t^{x_1}\nu_R\nu_Bu^{y_1}\nu_B^{-1}\nu_R^{-1}t^{x_2}\nu_R\nu_Bu^{y_2}\nu_B^{-1}\nu_R^{-1}\dots \nu_R\nu_Bu^{y_n}\nu_B^{-1}\nu_R^{-1}t^{x_n}
$$
for $x_i$, $y_i\in \mathbb{Z}\setminus\{0\}$ for $i=1,2,\dots,n$ where $n$ is a positive integer.
The $W_3^{\Delta_i}$ invariant of these barbells can be calculated following the method in  \cite{niu2025mappingclassgroup4dimensional} (cf. Lemma 5.9 of \cite{niu2025mappingclassgroup4dimensional}).
Namely, they can be written as linear combinations of $W_3$ of sub-barbells isotopic to the form of $a\nu_R\nu_B c$ or $a\nu_B\nu_Rc$ such that $(a,c)$ is admissible:

\begin{definition}[Admissible pairs]\label{def:adm}
A pair $(a,c)$ of nontrivial reduced words in $\langle t,u\rangle$ is \emph{admissible} if either
\begin{itemize}[leftmargin=2em]
\item $a$ ends with a nonzero power of $t$ and $c$ begins with a nonzero power of $u$, or
\item $a$ ends with a nonzero power of $u$ and $c$ begins with a nonzero power of $t$.
\end{itemize}
\end{definition}

\begin{proposition}\label{prop:admissible-span}
The set of values ${W_3^{\Delta_i}}$ on barbells in $p(\pi_1 \mathrm{Emb}(I,Y))$ is contained in the span of the ${W_3^{\Delta_i}}$ values on barbells in the form
of $a\nu_R\nu_B c$ or $a\nu_B\nu_Rc$ such that $(a,c)$ is admissible.
\end{proposition}

\section{Proof of Theorem A}\label{sec:computations}
We prove Theorem \ref{thm:A} by showing that the barbells $\Phi_{\mathcal{B}(t\nu_B\nu_Rtu^kt^{-1})}$, $k=1,2,\dots$ do not lie in the span of admissible barbells. We work in the free group
$$
F=\langle t_1,u_1,t_3,u_3\rangle
$$
and its rational group ring $\mathbb Q[F]$ (viewed additively). Let $\mathcal A$ denote the free $\mathbb Q$-vector space on
reduced words in $F$ (so $\mathcal A\cong \mathbb Q[F]$ additively). In particular, there are no commutation relations among
$t_1,u_1,t_3,u_3$, and equality of monomials always means equality of reduced words in $F$. 

We first observe that for the barbells $\Phi_{\mathcal{B}(t\nu_B\nu_Rtu^kt^{-1})}$, $k=1,2,\dots$, there is a single Type 4 intersection point between $\Delta_1$ and the bar, and two Type 4 intersection points plus a Type 6 intersection point between $\Delta_2$ and the bar (cf. Figures 14, 17 and 25 of \cite{niu2025mappingclassgroup4dimensional}).  Applying Lemma 5.10 and Lemma 5.11 of \cite{niu2025mappingclassgroup4dimensional} to the barbells $t\nu_B\nu_Rtu^kt^{-1}$ for
$k=1,2,\dots$, we have
\begin{align}
W_3^{\Delta_1}\bigl(\Phi_{\mathcal{B}(t\nu_B\nu_Rtu^kt^{-1})}\bigr)
&=\;T_4\!\left(t^{-1},\,tu^{-k}t^{-1}\right),
\label{eq:target1}\\
W_3^{\Delta_2}\bigl(\Phi_{\mathcal{B}(t\nu_B\nu_Rtu^kt^{-1})}\bigr)
&=\;2\,T_4\!\left(t^{-1},\,tu^{-k}t^{-1}\right)
\;+\;
T_6\!\left(t^{-1},\,tu^{-k}t^{-1}\right).
\label{eq:target2}
\end{align}
where
\begin{equation}\label{eq:T4expandedk}
\begin{aligned}
T_4\!\left(t^{-1},\,tu^{-k}t^{-1}\right)
=
\Bigl(&
t_1^{-1}t_3u_3^{-k}t_3^{-2}
+t_1t_3u_3^{-k}
-t_1u_1^{-k}t_1^{-1}t_3^{-1}
-t_1u_1^{-k}t_1^{-1}t_3
\\
&+
t_1^{-1}t_3u_3^{-k}t_3^{-1}
+t_1t_3u_3^{-k}t_3^{-1}
-t_1u_1^{-k}t_1^{-2}t_3^{-1}
-t_1u_1^{-k}t_3
\Bigr).
\end{aligned}
\end{equation}
and
\begin{equation}\label{eq:T6expandedk}
\begin{aligned}
T_6\!\left(t^{-1},\,tu^{-k}t^{-1}\right)=\;&
t_1u_1^{-k}t_1^{-1}\,u_3^{-k}t_3^{-1}
+t_1u_1^{k}t_1^{-1}\,u_3^{k}t_3^{-1}
-t_1t_3u_3^{k}
-t_1t_3u_3^{-k}
\\
&+
t_1^{-1}t_3u_3^{k}t_3^{-2}
+t_1^{-1}t_3u_3^{-k}t_3^{-2}
+t_1^{-1}t_3u_3^{-k}t_3^{-1}
+t_1^{-1}t_3u_3^{k}t_3^{-1}
\\
&-
t_1u_1^{-k}t_1^{-1}t_3^{-1}
-t_1u_1^{k}t_1^{-1}t_3^{-1}
-t_1u_1^{-k}t_1^{-2}t_3^{-1}
-t_1u_1^{k}t_1^{-2}t_3^{-1}
\\
&+
t_1u_1^{-k}t_3
+t_1u_1^{k}t_3
-u_1^{-k}t_1^{-1}t_3u_3^{-k}t_3^{-1}
-u_1^{k}t_1^{-1}t_3u_3^{k}t_3^{-1}.
\end{aligned}
\end{equation}

Let $\langle t,u\rangle$ be the free group on $t,u$. For any word $w\in\langle t,u\rangle$ define
$$
w_1:=w(t_1,u_1),\qquad w_3:=w(t_3,u_3),
$$
and write $\overline{w}=w^{-1}$ following the convention in \cite{niu2025mappingclassgroup4dimensional}. Recall from Theorem 3.16 of \cite{niu2025mappingclassgroup4dimensional} that for $\nu,\mu\in\langle t,u\rangle$ we have the
\textbf{hexagon relations}
\begin{equation}\label{eq:HexagonElement}
H(\nu,\mu)
:=\nu_1\mu_3\;+\;\mu_1^{-1}\nu_3^{-1}\;-\;\nu_1^{-1}\mu_3\nu_3^{-1}\;-\;\nu_1\mu_1^{-1}\nu_3
\in \mathcal A.
\end{equation}


Following the approach as in \cite{niu2025mappingclassgroup4dimensional}, we observe that the barbells described in Proposition \ref{prop:admissible-span} only produce Type 1, Type 3, Type 4 and Type 6 points under scanning, leading to polynomials
\begin{align*}
T_1(\overline{a},\overline{c})
&=
a_1\,\overline{c}_3\,a_3
+\overline{c}_1\,a_1\,a_3
-\overline{c}_1\,\overline{a}_3
-\overline{a}_1\,\overline{c}_3,
\\[1mm]
T_3(\overline{a},\overline{c})
&=
-\,c_1\,\overline{a}_3\,c_3
+a_1\,\overline{c}_3\,a_3
+\overline{a}_1\,c_3\,\overline{a}_3
-c_1\,a_3
-a_1\,c_3
\\[-1mm]
&\qquad
+c_1\,\overline{a}_1\,\overline{a}_3
+\overline{c}_1\,a_1\,a_3
-a_1\,\overline{c}_1\,\overline{c}_3,
\\[1mm]
T_4(\overline{a},\overline{c})
&=
\overline{a}_1\,\overline{c}_3\,\overline{a}_3
+a_1\,\overline{c}_3\,a_3
-\overline{c}_1\,\overline{a}_3
-\overline{c}_1\,a_3
+\overline{a}_1\,\overline{c}_3
+a_1\,\overline{c}_3
\\[-1mm]
&\qquad
-\overline{c}_1\,\overline{a}_1\,\overline{a}_3
-\overline{c}_1\,a_1\,a_3,
\\[1mm]
T_6(\overline{a},\overline{c})
&=
\overline{c}_1\,\overline{a}_3\,\overline{c}_3
+c_1\,\overline{a}_3\,c_3
-a_1\,c_3\,a_3
-a_1\,\overline{c}_3\,a_3
+\overline{a}_1\,c_3\,\overline{a}_3
+\overline{a}_1\,\overline{c}_3\,\overline{a}_3
\\[-1mm]
&\qquad
+\overline{a}_1\,\overline{c}_3
+\overline{a}_1\,c_3
-\overline{c}_1\,\overline{a}_3
-c_1\,\overline{a}_3
-\overline{c}_1\,\overline{a}_1\,\overline{a}_3
-c_1\,\overline{a}_1\,\overline{a}_3
\\[-1mm]
&\qquad
+\overline{c}_1\,a_1\,a_3
+c_1\,a_1\,a_3
-\overline{a}_1\,\overline{c}_1\,\overline{c}_3
-\overline{a}_1\,c_1\,c_3.
\end{align*}
in $\mathcal{A}$ as in Lemma 5.10 and Lemma 5.11 of \cite{niu2025mappingclassgroup4dimensional} (specialized to $b=1$ and written in terms of
$(\overline a,\overline c)$). In particular, the ${W_3^{\Delta_i}}$ of barbells in the image $p(\pi_1 \mathrm{Emb}(I,Y))$ is contained in $\mathcal S\le \mathcal A$, the $\mathbb Q$-subspace spanned by all $T_i(\overline a,\overline c)$ with $(a,c)$ admissible and
$i\in\{1,3,4,6\}$. Let $\mathcal H\le \mathcal A$ be the $\mathbb Q$-subspace spanned by all hexagon relations $H(\nu,\mu)$. Let $\pi:\mathcal A\to \mathcal A/\mathcal H$ be the quotient map, and write $\overline{\mathcal S}:=\pi(\mathcal S)$.

In the remainder of this section, we show that $W_3^{\Delta_i}\bigl(\Phi_{\mathcal{B}(t\nu_B\nu_Rtu^kt^{-1})}\bigr)$ is not in $\overline{\mathcal{S}}$ for $i=1,2$ and $k\in \mathbb{Z}^{+}$. Fix $k\ge 1$ and define
\begin{equation}\label{eq:m1m2k}
m_1(k):=t_1^{-1}t_3u_3^{-k}t_3^{-2},
\qquad
m_2(k):=t_1^{2}u_1^{k}t_1^{-1}t_3.
\end{equation}
Define a function $\Psi_k:\mathcal A\to\mathbb Q$ by
\begin{equation}\label{eq:PsiDef}
\Psi_k(x):=\mathrm{coeff}_{m_1(k)}(x)-\mathrm{coeff}_{m_2(k)}(x),
\end{equation}
where $\mathrm{coeff}_{w}(x)$ denotes the coefficient of the word $w\in F$ in $x\in\mathcal A$. Note that $\Psi_k$ is linear.

\begin{lemma}\label{lem:PsiTarget1}
$\Psi_k\!\left(T_4(t^{-1},tu^{-k}t^{-1})\right)=1$.
\end{lemma}
\begin{proof}
From \eqref{eq:T4expandedk}, we observe that $m_1(k)$ appears exactly once with coefficient $+1$ in
$T_4\!\left(t^{-1},tu^{-k}t^{-1}\right)$, while $m_2(k)$ does not appear at all. Hence $\Psi_k=1$.
\end{proof}

\begin{lemma}\label{lem:PsiTarget2}
$\Psi_k\!\left(2T_4(t^{-1},tu^{-k}t^{-1})+T_6(t^{-1},tu^{-k}t^{-1})\right)=3$.
\end{lemma}
\begin{proof}
By Lemma \ref{lem:PsiTarget1}, $\Psi_k\!\left(2T_4(t^{-1},tu^{-k}t^{-1})\right)=2$.
From \eqref{eq:T6expandedk}, we observe that $m_1(k)$ appears exactly once with coefficient $+1$ in
$T_6\!\left(t^{-1},tu^{-k}t^{-1}\right)$, while $m_2(k)$ does not appear. Hence $\Psi_k(T_6)=1$.
Thus $\Psi_k=2+1=3$.
\end{proof}

\begin{lemma}\label{lem:PsiKillsHex}
For all $\nu,\mu\in\langle t,u\rangle$, $$\Psi_k\bigl(H(\nu,\mu)\bigr)=0.$$ In other words, $\Psi_k$ vanishes on the hexagon relations.
\end{lemma}
\begin{proof}
The Hexagon relation $H(\nu,\mu)$ contains four terms
$$
H(\nu,\mu)=\underbrace{\nu_1\mu_3}_{(1)}+\underbrace{\mu_1^{-1}\nu_3^{-1}}_{(2)}-\underbrace{\nu_1^{-1}\mu_3\nu_3^{-1}}_{(3)}
-\underbrace{\nu_1\mu_1^{-1}\nu_3}_{(4)}.
$$
We check case by case when any of the four terms equals $m_1(k)$ or $m_2(k)$.
As we verify below, in each case, both $m_1(k)$ and $m_2(k)$ occur simultaneously with the same sign with the remaining two terms not equal to $m_1(k)$ or $m_2(k)$. Therefore,
$$
\mathrm{coeff}_{m_1(k)}\bigl(H(\nu,\mu)\bigr)=\mathrm{coeff}_{m_2(k)}\bigl(H(\nu,\mu)\bigr),
$$
for all $\mu, \nu\in \langle t, u\rangle$.

\begin{itemize}
    \item $(1)=m_1(k)$ if and only if $\nu=t^{-1}$ and $\mu=tu^{-k}t^{-2}$. In this case,
$$
(2)=\mu_1^{-1}\nu_3^{-1}=t_1^2u_1^{k}t_1^{-1}t_3=m_2(k).
$$
and
$$
(3)=\nu_1^{-1}\mu_3\nu_3^{-1}=t_1t_3u_3^{-k}t_3^{-1},
\qquad
(4)=\nu_1\mu_1^{-1}\nu_3=t_1u_1^{k}t_1^{-1}t_3^{-1},
$$
so both are not equal to $m_1(k)$ or $m_2(k)$.
\item 
$(1)=m_2(k)$ if and only $\nu=t^2u^{k}t^{-1}$ and $\mu=t$. In this case,
$$
(2)=\mu_1^{-1}\nu_3^{-1}=t_1^{-1}t_3u_3^{-k}t_3^{-2}=m_1(k).
$$
and
$$
(3)=\nu_1^{-1}\mu_3\nu_3^{-1}
=t_1u_1^{-k}t_1^{-2}\,t_3^2u_3^{-k}t_3^{-2},
\qquad
(4)=\nu_1\mu_1^{-1}\nu_3
=t_1^2u_1^{k}t_1^{-2}\,t_3^2u_3^{k}t_3^{-1},
$$
so both are not equal to $m_1(k)$ or $m_2(k)$.

\item $(3)=m_1(k)$, if and only if $\nu=t$ and $\mu=tu^{-k}t^{-1}$. In this case,
$$
(4)=\nu_1\mu_1^{-1}\nu_3=t_1^2u_1^{k}t_1^{-1}t_3=m_2(k).
$$
and
$$
(1)=\nu_1\mu_3=t_1t_3u_3^{-k}t_3^{-1},
\qquad
(2)=\mu_1^{-1}\nu_3^{-1}=t_1u_1^{k}t_1^{-1}t_3^{-1},
$$
so both are not equal to $m_1(k)$ or $m_2(k)$.

\item 
$(4)=m_1(k)$ if and only if $\nu=tu^{-k}t^{-2}$ and $\mu=t^2u^{-k}t^{-2}$. In this case,
$$
(3)=\nu_1^{-1}\mu_3\nu_3^{-1}=t_1^2u_1^{k}t_1^{-1}t_3=m_2(k).
$$
and 
$$
(1)=\nu_1\mu_3=t_1u_1^{-k}t_1^{-2}\,t_3^2u_3^{-k}t_3^{-2},
\qquad
(2)=\mu_1^{-1}\nu_3^{-1}=t_1^2u_1^{k}t_1^{-2}\,t_3^2u_3^{k}t_3^{-1},
$$
so both are not equal to $m_1(k)$ or $m_2(k)$.
\end{itemize}

Therefore we have $\Psi_k(H(\nu,\mu))=0$ for all $\mu, \nu \in\langle t,u \rangle$.
\end{proof}

 The following proposition is in the same spirit as Proposition 5.16 in \cite{niu2025mappingclassgroup4dimensional}.
\begin{proposition}\label{linearinde}
The families
$$
\Bigl\{T_4(t^{-1},tu^{-k}t^{-1})\Bigr\}_{k\ge 1}
$$
and 
$$
\Bigl\{2T_4(t^{-1},tu^{-k}t^{-1})+T_6(t^{-1},tu^{-k}t^{-1})\Bigr\}_{k\ge 1}
$$
are $\mathbb Q$-linearly independent.
\end{proposition}

\begin{proof}
Suppose 
$$
\sum_{j=1}^n q_j\,T_4\!\left(t^{-1},tu^{-j}t^{-1}\right)=0$$ in $\mathcal{A}/\mathcal H$.
For a fixed $k\in\{1,\dots,n\}$, applying $\Psi_k$ leads to 
$$
0
=\sum_{j=1}^n q_j\,\Psi_k\!\left(T_4(t^{-1},tu^{-j}t^{-1})\right).
$$
The term $m_1(k)=t_1^{-1}t_3u_3^{-k}t_3^{-2}$ appears in
$T_4(t^{-1},tu^{-j}t^{-1})$ if and only if $j=k$ with coefficient $+1$ and $m_2(k)$ never appears in any
$T_4(t^{-1},tu^{-j}t^{-1})$. Hence $\Psi_k\!\left(T_4(t^{-1},tu^{-j}t^{-1})\right)=\delta_{kj}$, thus $q_k=0$.
The same process shows that all $q_j$ vanish.

Similarly, suppose 
$$
\sum_{j=1}^n q_j\Bigl(2T_4(t^{-1},tu^{-j}t^{-1})+T_6(t^{-1},tu^{-j}t^{-1})\Bigr)=0$$ in $\mathcal{A}/\mathcal H$. Applying $\Psi_k$ gives rise to
$$
0=\sum_{j=1}^n q_j\,\Psi_k\!\left(2T_4(t^{-1},tu^{-j}t^{-1})+T_6(t^{-1},tu^{-j}t^{-1})\right).
$$
where $m_1(k)$ appears exactly when $j=k$ with 
coefficient $2+1=3$ and $m_2(k)$ never appears. Hence
$$
\Psi_k\!\left(2T_4(t^{-1},tu^{-j}t^{-1})+T_6(t^{-1},tu^{-j}t^{-1})\right)=3\,\delta_{kj},
$$
which implies that $q_k=0$ for all $k=1,\dots,n$.
\end{proof}

\begin{lemma}\label{lem:admTable}
For an admissible pair $(a,c)$ and $i\in\{1,3,4,6\}$,
$$
\mathrm{coeff}_{m_1(k)}\!\bigl(T_i(\overline a,\overline c)\bigr)=
\mathrm{coeff}_{m_2(k)}\!\bigl(T_i(\overline a,\overline c)\bigr)=0.
$$
Therefore $\Psi_k$ vanishes on $\mathcal{S}$.
\end{lemma}
\begin{proof}
In Table~\ref{tab:solutionsk} we list, for each distinct monomial term $M(a,c)$ in $T_1,T_3,T_4,T_6$,
the unique solution pair $(a,c)$ (if there is one) for which $M(a,c)=m_1(k)$ or
$M(a,c)=m_2(k)$. We observe that in every solution, $a$ ends with a nonzero power of $t$ and $c$ begins with a nonzero power of $t$.
Therefore none of these pairs is admissible by Definition~\ref{def:adm}.
Hence, for admissible $(a,c)$, no monomial term in $T_i(\overline a,\overline c)$ can equal $m_1(k)$ or $m_2(k)$,
thus both coefficients are $0$.
\end{proof}

\begin{center}
\setlength{\LTpre}{0pt}
\setlength{\LTpost}{0pt}
\setlength{\tabcolsep}{4pt}
\renewcommand{\arraystretch}{1.10}
\footnotesize
\newcolumntype{P}[1]{>{\raggedright\arraybackslash}p{#1}}
\begin{longtable}{@{}P{2.1cm} P{5.4cm} P{4.3cm} P{4.3cm}@{}}
\caption{Solutions to $M(a,c)=m_1(k)$ and $M(a,c)=m_2(k)$ for monomials $M(a,c)$ appearing in $T_1,T_3,T_4,T_6$.
All solutions are not admissible because $a$ ends with $t^{\pm1}$ and $c$ begins with $t^{\pm1}$.}\label{tab:solutionsk}\\
\toprule
appears in & monomial term $M(a,c)$ & $(a,c)$ if $M=m_1(k)$ & $(a,c)$ if $M=m_2(k)$ \\
\midrule
\endfirsthead
\toprule
appears in & monomial term $M(a,c)$ & $(a,c)$ if $M=m_1(k)$ & $(a,c)$ if $M=m_2(k)$ \\
\midrule
\endhead
\bottomrule
\endlastfoot

$T_{1,3,4,6}$ &
$a_1\,\overline{c}_3\,a_3$ &
$\bigl(t^{-1},\,t u^{k} t^{-1}\bigr)$ &
$\bigl(t^{2}u^{k}t^{-1},\,t^{2}u^{k}t^{-2}\bigr)$ \\

$T_{1,3,4,6}$ &
$\overline{c}_1\,a_1\,a_3$ &
$\bigl(tu^{-k}t^{-2},\,tu^{-k}t^{-1}\bigr)$ &
$\bigl(t,\,t^{2}u^{-k}t^{-2}\bigr)$ \\

$T_{1,4,6}$ &
$\overline{c}_1\,\overline{a}_3$ &
$\bigl(t^{2}u^{k}t^{-1},\,t\bigr)$ &
$\bigl(t^{-1},\,tu^{-k}t^{-2}\bigr)$ \\

$T_{1,4,6}$ &
$\overline{a}_1\,\overline{c}_3$ &
$\bigl(t,\,t^{2}u^{k}t^{-1}\bigr)$ &
$\bigl(tu^{-k}t^{-2},\,t^{-1}\bigr)$ \\

$T_{3,6}$ &
$c_1\,\overline{a}_3\,c_3$ &
$\bigl(tu^{k}t^{-1},\,t^{-1}\bigr)$ &
$\bigl(t^{2}u^{k}t^{-2},\,t^{2}u^{k}t^{-1}\bigr)$ \\

$T_{3,6}$ &
$\overline{a}_1\,c_3\,\overline{a}_3$ &
$\bigl(t,\,tu^{-k}t^{-1}\bigr)$ &
$\bigl(tu^{-k}t^{-2},\,t^{2}u^{-k}t^{-2}\bigr)$ \\

$T_{3}$ &
$c_1\,a_3$ &
$\bigl(tu^{-k}t^{-2},\,t^{-1}\bigr)$ &
$\bigl(t,\,t^{2}u^{k}t^{-1}\bigr)$ \\

$T_{3}$ &
$a_1\,c_3$ &
$\bigl(t^{-1},\,tu^{-k}t^{-2}\bigr)$ &
$\bigl(t^{2}u^{k}t^{-1},\,t\bigr)$ \\

$T_{3,6}$ &
$c_1\,\overline{a}_1\,\overline{a}_3$ &
$\bigl(t^{2}u^{k}t^{-1},\,tu^{k}t^{-1}\bigr)$ &
$\bigl(t^{-1},\,t^{2}u^{k}t^{-2}\bigr)$ \\

$T_{3}$ &
$a_1\,\overline{c}_1\,\overline{c}_3$ &
$\bigl(tu^{k}t^{-1},\,t^{2}u^{k}t^{-1}\bigr)$ &
$\bigl(t^{2}u^{k}t^{-2},\,t^{-1}\bigr)$ \\

$T_{4,6}$ &
$\overline{a}_1\,\overline{c}_3\,\overline{a}_3$ &
$\bigl(t,\,tu^{k}t^{-1}\bigr)$ &
$\bigl(tu^{-k}t^{-2},\,t^{2}u^{k}t^{-2}\bigr)$ \\

$T_{4,6}$ &
$\overline{c}_1\,a_3$ &
$\bigl(tu^{-k}t^{-2},\,t\bigr)$ &
$\bigl(t,\,tu^{-k}t^{-2}\bigr)$ \\

$T_{4}$ &
$a_1\,\overline{c}_3$ &
$\bigl(t^{-1},\,t^{2}u^{k}t^{-1}\bigr)$ &
$\bigl(t^{2}u^{k}t^{-1},\,t^{-1}\bigr)$ \\

$T_{4,6}$ &
$\overline{c}_1\,\overline{a}_1\,\overline{a}_3$ &
$\bigl(t^{2}u^{k}t^{-1},\,tu^{-k}t^{-1}\bigr)$ &
$\bigl(t^{-1},\,t^{2}u^{-k}t^{-2}\bigr)$ \\

$T_{6}$ &
$c_1\,\overline{a}_3$ &
$\bigl(t^{2}u^{k}t^{-1},\,t^{-1}\bigr)$ &
$\bigl(t^{-1},\,t^{2}u^{k}t^{-1}\bigr)$ \\

$T_{6}$ &
$c_1\,a_1\,a_3$ &
$\bigl(tu^{-k}t^{-2},\,tu^{k}t^{-1}\bigr)$ &
$\bigl(t,\,t^{2}u^{k}t^{-2}\bigr)$ \\

$T_{6}$ &
$\overline{c}_1\,\overline{a}_3\,\overline{c}_3$ &
$\bigl(tu^{k}t^{-1},\,t\bigr)$ &
$\bigl(t^{2}u^{k}t^{-2},\,tu^{-k}t^{-2}\bigr)$ \\

$T_{6}$ &
$a_1\,c_3\,a_3$ &
$\bigl(t^{-1},\,tu^{-k}t^{-1}\bigr)$ &
$\bigl(t^{2}u^{k}t^{-1},\,t^{2}u^{-k}t^{-2}\bigr)$ \\

$T_{6}$ &
$\overline{a}_1\,\overline{c}_1\,\overline{c}_3$ &
$\bigl(tu^{-k}t^{-1},\,t^{2}u^{k}t^{-1}\bigr)$ &
$\bigl(t^{2}u^{-k}t^{-2},\,t^{-1}\bigr)$ \\

$T_{6}$ &
$\overline{a}_1\,c_3$ &
$\bigl(t,\,tu^{-k}t^{-2}\bigr)$ &
$\bigl(tu^{-k}t^{-2},\,t\bigr)$ \\

$T_{6}$ &
$\overline{a}_1\,c_1\,c_3$ &
$\bigl(tu^{-k}t^{-1},\,tu^{-k}t^{-2}\bigr)$ &
$\bigl(t^{2}u^{-k}t^{-2},\,t\bigr)$ \\

\end{longtable}
\normalsize
\end{center}

\begin{theorem}\label{thm:main-k}
For every integer $k\ge 1$,
\[W_3^{\Delta_1}\bigl(\Phi_{\mathcal{B}(t\nu_B\nu_Rtu^kt^{-1})}\bigr)=
T_4(t^{-1},tu^{-k}t^{-1})\notin \overline{\mathcal S}
\]
and
\[W_3^{\Delta_2}\bigl(\Phi_{\mathcal{B}(t\nu_B\nu_Rtu^kt^{-1})}\bigr)=
2T_4(t^{-1},tu^{-k}t^{-1})+T_6(t^{-1},tu^{-k}t^{-1})\notin \overline{\mathcal S}.
\]
\end{theorem}

\begin{proof}
This follows from Lemma \ref{lem:PsiTarget1} 
Lemma \ref{lem:PsiTarget2}, Lemma \ref{lem:PsiKillsHex} and Lemma \ref{lem:admTable}.
\end{proof}
In particular, for $i=1,2$ and $k\in \mathbb{Z}^{+}$, the family 
\[\bigl\{W_3'^{\Delta_i}\bigl(\Phi_{\mathcal{B}(t\nu_B\nu_Rtu^kt^{-1})}\bigr)\bigr \},\]
are non-zero and linearly independent.

Combining Proposition \ref{prop:admissible-span}, Proposition~\ref{linearinde} and Theorem~\ref{thm:main-k} gives Theorem~\ref{thm:A}.

\section{Proof of Theorem \ref{cor:B}}
\label{thbproof}
We prove Theorem \ref{cor:B} in this section.

\begin{proof}[Proof of Theorem \ref{cor:B}]

Let $[\Phi_k]\coloneqq [\Phi_{\mathcal{B}(t\nu_B\nu_R tu^kt^{-1})}]\in \pi_0\mathrm{Diff}(Y,\partial)$.
Assume for contradiction that the pair $(\Phi_k(\Delta_1),\Phi_k(\Delta_2))$ is isotopic to $(\Delta_1,\Delta_2)$.
Choose an isotopy $F\colon (D^3,D^3)\times I \to Y$ with $F_0=(\Delta_1,\Delta_2)$ and $F_1=(\Phi_k(\Delta_1),\Phi_k(\Delta_2))$.
By isotopy extension, there is an ambient isotopy $\overline{F}\colon Y\times I\to Y$
with $\overline{F}_0=\mathrm{Id}$ and $\overline{F}_1(\Delta_1,\Delta_2)=(\Phi_k(\Delta_1),\Phi_k(\Delta_2))$.
Then $\overline{F}_1^{-1}\circ \Phi_k$ is isotopic to $\Phi_k$ (rel $\partial Y$) and fixes $(\Delta_1,\Delta_2)$ pointwise.
Replacing $\Phi_k$ by this isotopic representative, we assume from now on that $\Phi_k|_{\Delta_i}=\mathrm{Id}$. By a further local straightening (supported in a small tubular neighbourhood of $(\Delta_1,\Delta_2)$), we may moreover assume that
$\Phi_k$ fixes an open product neighbourhood of $(\Delta_1,\Delta_2) $ pointwise. Set
\[
W\coloneqq Y\setminus \nu(\Delta_1\cup \Delta_2).
\]
Then $\Phi_k$ restricts to an element of $\mathrm{Diff}(W,\partial W)$ fixing all boundary components pointwise. 

Cap off the two boundary components of $Y$ 
by gluing in $S^2\times D^2$. Denote the resulting manifold by $M$.
By construction, $M$ is diffeomorphic to $S^4$, and the extension of $\Phi_k$ by the identity over
the attached $S^2\times D^2$ is isotopic to $\mathrm{Id}$ in $\mathrm{Diff}(M)$.

There are embedded $4$--disks $B_1,B_2\subset \mathrm{Int}(M)$ with
\[
W \cong M\setminus \operatorname{int}(B_1\cup B_2),
\]
and under this identification $\Phi_k|_W$ fixes $\partial (B_1\cup B_2)$ pointwise.
Equivalently, extending $\Phi_k|_W$ by the identity over $B_1$ and $B_2$ gives a diffeomorphism
\[
\widehat{\Phi}_k \in \mathrm{Diff}(M\ \mathrm{rel}\ B_1,B_2),
\]
whose image in $\pi_0\mathrm{Diff}(M)$ is trivial. Let $B=(B_1,B_2)$. The mapping class $[\widehat{\Phi}_k]$ lies in the kernel
\[
DS_B(M)\coloneqq \ker\bigl(\pi_0\mathrm{Diff}(M\ \mathrm{rel}\ B)\to \pi_0\mathrm{Diff}(M)\bigr),
\]
which Lucas calls the \emph{disk slide group} (see \cite[\S2.1]{lucas2025isotopyversusequivariantisotopy}). By Lucas' construction, there is a natural surjection
\[
DS_B \colon \pi_1(\mathrm{Fr}_B(M))\twoheadrightarrow DS_B(M)
\]
(\cite[Proposition~2.2]{lucas2025isotopyversusequivariantisotopy}) where $\mathrm{Fr}_B(M)\subset\mathrm{Conf}_2(\mathrm{Fr}(M))$ is the sub-bundle consisting of pairs of frames over two distinct points.
Choose $\gamma\in \pi_1(\mathrm{Fr}_B(M))$ with $DS_B(\gamma)=[\widehat{\Phi}_k]$.
Lucas also defines a map
\[
\rho\colon \pi_1(\mathrm{Fr}_B(M))\to \pi_1(M,p)
\]
and proves that there is an exact sequence
\[
\pi_1(SO(4)) \longrightarrow \pi_1(\mathrm{Fr}_B(M)) \xrightarrow{\rho} \pi_1(M,p) \longrightarrow 1,
\]
together with the identification that the image of $\pi_1(SO(4))$ under $DS_B$ is generated by the boundary sphere twist
(\cite[Proposition~2.6 (i)]{lucas2025isotopyversusequivariantisotopy}).




Since $\pi_1(M,p)$ is trivial, the class $[\widehat{\Phi}_k]=DS_B(\gamma)$ lies in the subgroup generated by the boundary sphere twist about $\partial B$.
Since $\pi_1(SO(4))\cong \mathbb{Z}/2$, this subgroup has order dividing $2$; hence
\[
[\widehat{\Phi}_k]^2=1 \quad\text{in}\quad \pi_0\mathrm{Diff}(M\ \mathrm{rel}\ B).
\]
Restricting back to $W=M\setminus \operatorname{int}(B)$ and then viewing $W\subset Y$, we conclude that
$\Phi_k^2$ is isotopic to the identity on $W$ (rel $\partial W$).

Finally, since $\Phi_k$ fixes $\nu(\Delta_i)$ pointwise, we can extend this isotopy over $\nu(\Delta_i)$ by the identity,
obtaining an isotopy of $\Phi_k^2$ to $\mathrm{Id}$ in $\mathrm{Diff}(Y,\partial Y)$.
Hence $[\Phi_k]^2=1$ in $\pi_0\mathrm{Diff}(Y,\partial)$. However, $[\Phi_k]$ has infinite order: since $W_3^{\Delta_i}$ is a homomorphism and we have \[W_3'^{\Delta_i}([\Phi_k]^n) = nW_3'^{\Delta_i}([\Phi_k])\]
for all $n\in \mathbb{Z}^{+}\neq 0$.
\end{proof}
\clearpage
\bibliographystyle{abbrv}
\bibliography{sample}

@misc{Budney-gabai,
      title={Knotted 3-balls in {$S^4$}}, 
      author={Ryan Budney and David Gabai},
      year={2021},
      note={arXiv preprint: 1912.09029},
      archivePrefix={arXiv},
      primaryClass={math.GT},
      url={https://arxiv.org/abs/1912.09029}, 
}

@misc{Budney-gabai2,
      title={On the automorphism groups of hyperbolic manifolds}, 
      author={Ryan Budney and David Gabai},
      year={2023},
 note={arXiv preprint: 2303.05010},
      eprint={2303.05010},
      archivePrefix={arXiv},
      primaryClass={math.GT},
      url={https://arxiv.org/abs/2303.05010}, 
}

@article{danica,
  author  = {Kosanovi{\'c}, Danica},
  title   = {On homotopy groups of spaces of embeddings of an arc or a circle: the Dax invariant},
  journal = {Transactions of the American Mathematical Society},
  volume  = {377},
  number  = {2},
  pages   = {775--805},
  year    = {2024},
  doi     = {10.1090/tran/8805}
}

@misc{niu2025mappingclassgroup4dimensional,
      title={On the mapping class group of 4-dimensional 1-handlebodies via {B}udney-{G}abai invariants}, 
      author={Weizhe Niu},
      year={2025},
 note={arXiv preprint: 2512.15099},
      eprint={2512.15099},
      archivePrefix={arXiv},
      primaryClass={math.GT},
      url={https://arxiv.org/abs/2512.15099}, 
}

@misc{lucas2025isotopyversusequivariantisotopy,
      title={Isotopy versus equivariant isotopy in dimensions three and higher}, 
      author={Trent Lucas},
      year={2025},
      note={arXiv preprint: 2508.11104},
      eprint={2508.11104},
      archivePrefix={arXiv},
      primaryClass={math.GT},
      url={https://arxiv.org/abs/2508.11104}, 
}

@misc{powell2025spanning3discs4spherepushed,
      title={Spanning 3-discs in the 4-sphere pushed into the 5-disc}, 
      author={Mark Powell},
      year={2025},
        note={arXiv preprint: 2512.05952},
      eprint={2512.05952},
      archivePrefix={arXiv},
      primaryClass={math.GT},
      url={https://arxiv.org/abs/2512.05952}, 
}

@misc{lin2025mappingclassgroups4manifolds,
      title={On the mapping class groups of 4-manifolds with 1-handles}, 
      author={Jianfeng Lin and Yi Xie and Boyu Zhang},
      year={2025},
              note={arXiv preprint: 2501.11821},
      eprint={2501.11821},
      archivePrefix={arXiv},
      primaryClass={math.GT},
      url={https://arxiv.org/abs/2501.11821}, 
}

\bigskip
\noindent\textsc{Yau Mathematical Sciences Center, Tsinghua University}\\
\textit{Email:} \texttt{weizheniu@mail.tsinghua.edu.cn}

\end{document}